\documentclass[10pt]{amsart}
\usepackage{amsmath, amsthm, amscd, amsfonts, amssymb, graphicx, color, hyperref}
\usepackage[all]{xy}

\newtheorem{theorem}{Theorem}[section]
\newtheorem{lemma}[theorem]{Lemma}
\newtheorem{proposition}[theorem]{Proposition}

\theoremstyle{definition}
\newtheorem{definition}[theorem]{Definition}
\newtheorem{example}[theorem]{Example}

\theoremstyle{remark}

\numberwithin{equation}{section}

\DeclareMathOperator{\bbH}{\mathbb{H}}
\DeclareMathOperator{\bbI}{\mathbb{I}}
\DeclareMathOperator{\bbN}{\mathbb{N}}

\DeclareMathOperator{\bbR}{\mathbb{R}}

\DeclareMathOperator{\A}{\mathcal{A}}
\DeclareMathOperator{\B}{\mathcal{B}}

\DeclareMathOperator{\card}{\operatorname{card}}
\DeclareMathOperator{\D}{\mathcal{D}}
\DeclareMathOperator{\diam}{\operatorname{diam}}
\DeclareMathOperator{\F}{\mathcal{F}}

\DeclareMathOperator{\id}{\operatorname{id}}
\DeclareMathOperator{\J}{\mathcal{J}}

\renewcommand{\L}{\mathcal{L}}

\DeclareMathOperator{\X}{\mathcal{X}}

\begin{document}

\title[A Generalization of Whyburn's Theorem]{A Generalization of Whyburn's Theorem, and Aperiodicity for Abelian $C^*$-Inclusions}

\author[Vrej Zarikian]{Vrej Zarikian}
\address{Department of Mathematics, U. S. Naval Academy, Annapolis, MD 21402, USA.}
\email{zarikian@usna.edu}

\subjclass[2010]{Primary 54E40; Secondary 46J10.}


\keywords{irreducible map, almost one-to-one map, quasicontinuous selection, abelian $C^*$-algebra, unique pseudo-expectation, almost extension property, aperiodicity}

\date{Received: xxxxxx; Revised: yyyyyy; Accepted: zzzzzz.}

\begin{abstract}
Let $j:Y \to X$ be a continuous surjection of compact metric spaces. Whyburn proved that $j$ is \emph{irreducible}, meaning that $j(F) \subsetneq X$ for any proper closed subset $F \subsetneq Y$, if and only if $j$ is \emph{almost one-to-one}, in the sense that
\[
    \overline{\{y \in Y: j^{-1}(j(y)) = y\}} = Y.
\]
In this note we prove the following generalization: There exists a unique minimal closed set $K \subseteq Y$ such that $j(K) = X$ if and only if
\[
    \overline{\{x \in X: \card(j^{-1}(x)) = 1\}} = X.
\]
Translated to the language of operator algebras, this says that if $\A \subseteq \B$ is a unital inclusion of separable abelian $C^*$-algebras, then there exists a unique \emph{pseudo-expectation} (in the sense of Pitts) if and only if the \emph{almost extension property} of Nagy-Reznikoff holds. More generally, we prove that a unital inclusion of (not necessarily separable) abelian $C^*$-algebras has a unique pseudo-expectation if and only if it is \emph{aperiodic} (in the sense of Kwa\'{s}niewski-Meyer).
\end{abstract}
\maketitle

\section{Introduction}

Let $j:Y \to X$ be a continuous surjection of compact Hausdorff spaces. Following \cite{BOT2006}, we say that $j$ is \emph{almost one-to-one} if
\[
    \overline{\{y \in Y: j^{-1}(j(y)) = y\}} = Y.
\]
We say that $j$ is \emph{irreducible} if $j(F) \subsetneq X$ for all proper closed subsets $F \subsetneq Y$. It is easy to see that
\[
    \text{$j$ one-to-one} \implies \text{$j$ almost one-to-one} \implies \text{$j$ irreducible}.
\]
In general, an irreducible map need not be almost one-to-one (see Example \ref{split interval} below). If, however, $Y$ is metrizable, then Whyburn's Theorem \cite[Theorem 2]{Whyburn1939} states that irreducible maps are necessarily almost one-to-one, so that the two notions coincide.\\

Irreducible maps and almost one-to-one maps arise naturally in topological dynamics. Indeed, a minimal self-map of a compact Hausdorff space must be irreducible \cite[Lemma 2.1]{KST2001}. Thus a minimal self-map of a compact metric space must be almost one-to-one, by Whyburn's Theorem (see also \cite[Theorem 2.7]{KST2001}).\\

In Theorem \ref{generalized Whyburn} below, we generalize Whyburn's Theorem, proving that for a continuous surjection $j:Y \to X$ of compact metric spaces, there exists a unique minimal closed set $K \subseteq Y$ such that $j(K) = X$ if and only if
\[
    \overline{\{x \in X: \card(j^{-1}(x)) = 1\}} = X.
\]
To see that this is a generalization, suppose that $j$ is irreducible. Then there exists a unique minimal closed set $K \subseteq Y$ such that $j(K) = X$, namely $K = Y$. By Theorem \ref{generalized Whyburn}, $\overline{X_1} = X$, where
\[
    X_1 = \{x \in X: \card(j^{-1}(x)) = 1\}.
\]
If
\[
    Y_1 = \{y \in Y: j^{-1}(j(y)) = y\},
\]
then $j(Y_1) = X_1$ and $j(\overline{Y_1}) = \overline{X_1} = X$. Thus $\overline{Y_1} = Y$, by irreducibility. So $j$ is almost one-to-one.\\

As the following example shows, Theorem \ref{generalized Whyburn} is a strict generalization of Whyburn's Theorem.

\begin{example}
Let
\[
    Y = ([0,1] \times \{0\}) \cup (\{1\} \times [0,1]) \subseteq \bbR^2,
\]
\[
    X = [0,1] \subseteq \bbR,
\]
and $j:Y \to X$ be defined by
\[
    j(x,y) = x.
\]
Then there exists a unique minimal closed set $K \subseteq Y$ such that $j(K) = X$, namely
\[
    K = [0,1] \times \{0\},
\]
and
\[
    \overline{\{x \in X: \card(j^{-1}(x)) = 1\}} = \overline{[0,1)} = [0,1] = X.
\]
Clearly $j$ is not irreducible. It is not almost one-to-one because
\[
    \overline{\{(x,y) \in Y: j^{-1}(j(x,y)) = (x,y)\}} = \overline{[0,1) \times \{0\}}
    = [0,1] \times \{0\} \subsetneq Y.
\]
\end{example}

When translated into the language of operator algebras, Theorem \ref{generalized Whyburn} says that for a unital inclusion $\A \subseteq \B$ of separable abelian $C^*$-algebras, there exists a unique \emph{pseudo-expectation} (in the sense of Pitts) if and only if the \emph{almost extension property} of Nagy-Reznikoff holds. (All these terms will be defined in Section \ref{opalg} below.) This answers a question posed to the author by Ruy Exel and Bartosz Kwa\'{s}niewski following his 2019 talk at the International Workshop on Operator Theory and its Applications in Lisbon, Portugal.\\

More generally, we prove (Theorem \ref{aperiodic}) that a unital inclusion $\A \subseteq \B$ of (not necessarily separable) abelian $C^*$-algebras has a unique pseudo-expectation if and only if the inclusion is \emph{aperiodic} (in the sense of Kwa\'{s}niewski and Meyer). Aperiodicity is weaker than the almost extension property \cite[Theorem 5.5]{KM2020}, but stronger than having a unique pseudo-expectation \cite[Theorem 3.6]{KM2020}. For separable $C^*$-inclusions, aperiodicity and the almost extension property coincide \cite[Theorem 5.5]{KM2020}. Thus Theorem \ref{generalized Whyburn} is actually a corollary of Theorem \ref{aperiodic}, modulo a non-trivial operator-algebraic result, as illustrated in the diagram below:
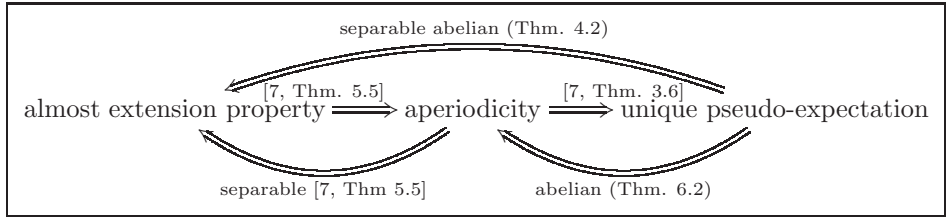
\[
\boxed{
\xymatrix{
    \text{almost extension property} \ar@{=>}[r]^{\text{\cite[Thm. 5.5]{KM2020}}} & \text{aperiodicity} \ar@{=>}[r]^{\text{\cite[Thm. 3.6]{KM2020}}} \ar@/^2pc/@{=>}[l]^(0.5){\text{separable \cite[Thm 5.5]{KM2020}}} & \text{unique pseudo-expectation} \ar@/^2pc/@{=>}[l]^{\text{abelian (Thm. \ref{aperiodic})}} \ar@/_2pc/@{=>}[ll]_{\text{separable abelian (Thm. \ref{generalized Whyburn})}}
}}
\]
Nonetheless, we include a direct, purely-topological proof of Theorem \ref{generalized Whyburn} that does not rely on Theorem \ref{aperiodic} nor on \cite[Theorem 5.5]{KM2020}.\\

\textbf{Acknowledgements:} We thank Alexis Alevras and Marcos Valdes for helpful discussions about this material.

\section{Preliminaries}

Let $j:Y \to X$ be a continuous surjection of compact Hausdorff spaces. In this section we show that there must exist a minimal closed set $K \subseteq Y$ such that $j(K) = X$ (Proposition \ref{minimal exists}) and we establish a criterion for determining when $K$ is unique (Proposition \ref{unique minimal}).

\begin{proposition} \label{minimal exists}
Let $j:Y \to X$ be a continuous surjection of compact Hausdorff spaces. Then there exists a minimal closed set $K \subseteq Y$ such that $j(K) = X$.
\end{proposition}

\begin{proof}
Let
\[
    \F = \{F \subseteq Y: \text{$F$ is closed and $j(F) = X$}\},
\]
a nonempty set partially ordered by reverse inclusion. Suppose $\L \subseteq \F$ is a nonempty totally ordered subset and define
\[
    F_0 = \bigcap\{F: F \in \L\}.
\]
If $x \in X$, then
\[
    \{F \cap j^{-1}(x): F \in \L\}
\]
is a collection of closed subsets of $Y$ with the finite intersection property. Since $Y$ is compact,
\[
    F_0 \cap j^{-1}(x) = \bigcap\{F \cap j^{-1}(x): F \in \L\} \neq \emptyset.
\]
Since $x$ was arbitrary, $j(F_0) = X$. So $F_0 \in \F$ is an upper bound for $\L$. By Zorn's Lemma, $\F$ has a maximal element.
\end{proof}

\begin{proposition} \label{unique minimal}
Let $j:Y \to X$ be a continuous surjection of compact Hausdorff spaces. Then the following are equivalent:
\begin{enumerate}
\item[i.] There exists a unique minimal closed set $K \subseteq Y$ such that $j(K) = X$.
\item[ii.] If $F_0 = \bigcap\{F \subseteq Y: \text{$F$ is closed and $j(F) = X$}\}$, then $j(F_0) = X$.
\end{enumerate}
\end{proposition}

\begin{proof}
Let
\[
    F_0 = \bigcap\{F \subseteq Y: \text{$F$ is closed and $j(F) = X$}\}.
\]
(i $\implies$ ii) Suppose there exists a unique minimal closed set $K \subseteq Y$ such that $j(K) = X$. If $F \subseteq Y$ is closed and $j(F) = X$, then $K \subseteq F$. Thus $K \subseteq F_0$. In fact, $F_0 = K$, since $F_0 \subseteq K$ by definition. Therefore $j(F_0) = j(K) = X$. (ii $\implies$ i) Conversely, suppose $j(F_0) = X$. Let $K \subseteq Y$ be a minimal closed set such that $j(K) = X$. Then $F_0 \subseteq K$ by definition and so $K = F_0$ by minimality.
\end{proof}

\section{Quasicontinuous Selections}

Let $j:Y \to X$ be a surjection. A \emph{selection} of $j$ is a function $\alpha:X \to Y$ such that $j \circ \alpha = \id_X$. It is well-known that continuous surjections of compact metric spaces have Borel selections (see \cite[Theorem 3.4.1.]{Arveson1976}, for example). On the other hand, a continuous selection is usually impossible. In this section, we draw attention to a theorem of Crannell, Frantz, and LeMasurier, which states that continuous surjections of compact metric spaces have \emph{quasicontinuous} selections. This result plays a critical role in our proof of Theorem \ref{generalized Whyburn} below.

\begin{definition}[quasiopen]
Let $Y$ be a topological space. We say that $A \subseteq Y$ is \textbf{quasiopen} if $A \subseteq \overline{A^\circ}$.
\end{definition}

\begin{definition}[quasicontinuous]
We say that a mapping $j:Y \to X$ of topological spaces is \textbf{quasicontinuous} if for every open set $U \subseteq X$, $j^{-1}(U) \subseteq Y$ is quasiopen.
\end{definition}

\begin{theorem}[\cite{CFL2005}, Corollary 5]
Let $j:Y \to X$ be a continuous surjection of compact metric spaces. Then $j$ admits a quasicontinuous selection $\alpha:X \to Y$.
\end{theorem}

\section{A Generalization of Whyburn's Theorem}

In this section we state and prove our main result, Theorem \ref{generalized Whyburn}. We start with an elementary lemma.

\begin{lemma} \label{USC}
Let $j:Y \to X$ be a continuous surjection of compact metric spaces. Then the mapping
\[
    X \to \bbR: x \mapsto \diam(j^{-1}(x))
\]
is upper semi-continuous.
\end{lemma}

\begin{proof}
Let $\varepsilon > 0$. Suppose that $x_n \to x$, where $\diam(j^{-1}(x_n)) \geq \varepsilon$ for all $n \in \bbN$. For each $n \in \bbN$, there exist $y_n, y_n' \in j^{-1}(x_n)$ such that $d(y_n,y_n') > \varepsilon-1/n$. Passing to a subsequence, if necessary, we may assume that $y_n \to y$ and $y_n' \to y'$. Then $y, y' \in j^{-1}(x)$ and $d(y,y') \geq \varepsilon$. Thus $\diam(j^{-1}(x)) \geq \varepsilon$.
\end{proof}

\begin{theorem} \label{generalized Whyburn}
Let $j:Y \to X$ be a continuous surjection of compact metric spaces. Then the following are equivalent:
\begin{enumerate}
\item[i.] There exists a unique minimal closed set $K \subseteq Y$ such that $j(K) = X$.
\item[ii.] $\{x \in X: \card(j^{-1}(x)) = 1\}$ is dense in $X$.
\end{enumerate}
\end{theorem}

\begin{proof}
(i $\implies$ ii) Suppose there exists a unique minimal closed set $K \subseteq Y$ such that $j(K) = X$. Let $\varepsilon > 0$. Assume that $\{x \in X: \diam(j^{-1}(x)) < \varepsilon\}$ is not dense in $X$. Then there exists $x_0 \in X$ and $r > 0$ such that
\[
    B_r(x_0) \subseteq \{x \in X: \diam(j^{-1}(x)) \geq \varepsilon\}.
\]
Let $\alpha:X \to Y$ be a quasicontinuous selection of $j$. Then $\alpha(x_0) \in j^{-1}(B_r(x_0))$ and so there exists $\delta > 0$ such that $B_\delta(\alpha(x_0)) \subseteq j^{-1}(B_r(x_0))$. Then $\alpha^{-1}(B_\delta(\alpha(x_0))$ is a quasiopen set containing $x_0$. Thus
\[
    U = \alpha^{-1}(B_\delta(\alpha(x_0)))^\circ
\]
is a nonempty open subset of $X$. Now
\[
    \alpha^{-1}(B_\delta(\alpha(x_0))) = j(B_\delta(\alpha(x_0)) \cap \alpha(X)),
\]
and so there exists $A \subseteq B_\delta(\alpha(x_0)) \cap \alpha(X)$ such that $j(A) = U$. In particular, $\diam(A) \leq 2\delta$. On the other hand,
\[
    U = j(A) \subseteq j(B_\delta(\alpha(x_0))) \subseteq B_r(x_0) \subseteq \{x \in X: \diam(j^{-1}(x)) \geq \varepsilon\}.
\]
Thus for every $y \in A$, there exists $y' \in Y$ such that $j(y') = j(y)$ and $d(y',y) > \varepsilon/4$. It follows that $d(y',A) \geq \varepsilon/4-2\delta$. Setting $A' = \{y': y \in A\}$, we see that $j(A') = U$ and $d(A',A) \geq \varepsilon/4-2\delta$. Choosing $\delta < \varepsilon/16$, we conclude that $d(A',A) > \varepsilon/8$, which implies $\overline{A} \cap \overline{A'} = \emptyset$. But then $F = \overline{A} \cup j^{-1}(U^c)$ and $F' = \overline{A'} \cup j^{-1}(U^c)$ are closed sets such that $j(F) = X$, $j(F') = X$, and $j(F \cap F') = U^c \subsetneq X$, which contradicts Proposition \ref{unique minimal}. By Lemma \ref{USC}, $\{x \in X: \diam(j^{-1}(x)) < \varepsilon\}$ is an open dense subset of $X$. Thus
\[
    \{x \in X: \card(j^{-1}(x)) = 1\}
    = \bigcap_{n=1}^\infty \{x \in X: \diam(j^{-1}(x)) < 1/n\}
\]
is dense in $X$.

(ii $\implies$ i) Suppose $X_1 = \{x \in X: \card(j^{-1}(x)) = 1\}$ is dense in $X$. Define $K = \overline{j^{-1}(X_1)}$. If $F \subseteq Y$ is closed and $j(F) = X$, then $j^{-1}(X_1) \subseteq F$, which implies $K \subseteq F$. Since
\[
    X = \overline{X_1} = \overline{j(j^{-1}(X_1))} \subseteq \overline{j(K)} = j(K),
\]
we see that $K$ is the unique minimal closed subset of $Y$ such that $j(K) = X$.
\end{proof}

\section{The Split Interval}

In Whyburn's theorem, the assumption of metrizability cannot be omitted \cite[Exercise 3.1.C]{Engelking1977}. Of course, the same is true for Theorem \ref{generalized Whyburn}, as the following example illustrates.

\begin{example} \label{split interval}
Let $\ddot{\bbI}$ be the ``split interval'', i.e., the set
\[
    ((0,1] \times \{0\}) \cup ([0,1) \times \{1\}) \subseteq \bbR^2
\]
equipped with the order topology arising from the lexicographic order
\[
    (x,y) < (x',y') \iff (x < x') \vee ((x = x') \wedge (y < y')).
\]
Let $\bbI = [0,1]$ be the unit interval equipped with its usual topology. Then $j:\ddot{\bbI} \to \bbI$ defined by
\[
    j(x,y) = x
\]
is an irreducible continuous surjection of compact Hausdorff spaces, but
\[
    \{x \in \bbI: \card(j^{-1}(x)) = 1\} = \{0, 1\},
\]
which is not dense in $\bbI$. Of course, $\ddot{\bbI}$ is not metrizable (it is not second countable).
\end{example}

\section{Aperiodicity for Abelian $C^*$-Inclusions} \label{opalg}

Let $\A \subseteq \B$ be a unital inclusion of $C^*$-algebras (cf. \cite{Arveson1976}). Then $\A \subseteq \B$ has the \emph{pure extension property} (PEP) if every pure state on $\A$ extends uniquely to a pure state on $\B$. A useful relaxation of the PEP is the \emph{almost extension property} (AEP) of Nagy and Reznikoff \cite{NagyReznikoff2014}. Instead of insisting that every pure state on $\A$ extends uniquely to a pure state on $\B$, the requirement is that a weak-* dense collection of pure states on $\A$ extends uniquely to pure states on $\B$. A \emph{conditional expectation} (CE) of $\B$ onto $\A$ is a completely positive map $E:\B \to \A$ such that $E|_{\A} = \id_{\A}$. A \emph{pseudo-expectation} (PsExp) for the $C^*$-inclusion $\A \subseteq \B$ is a completely positive map $\theta:\B \to I(\A)$ such that $\theta|_{\A} = \id_{\A}$ \cite[Definition 1.3]{Pitts2017}. Here $I(\A)$ is Hamana's \emph{injective envelope} of $\A$, the minimal injective operator system containing $\A$. Every conditional expectation is a pseudo-expectation, but there are $C^*$-inclusions with no conditional expectations, while pseudo-expectations always exist. We have the relations
\begin{equation} \label{implications}
    \text{PEP} \implies \text{AEP} \implies \text{$\exists!$ PsExp} \implies \text{$\exists$ at most one CE}.
\end{equation}
Only the middle implication is not immediate from the definitions.\\

It is somewhat surprising that the differences between the conditions in (\ref{implications}) above, in particular the difference between the almost extension property and having a unique pseudo-expectation, can be witnessed by abelian $C^*$-inclusions. Indeed, let $\A \subseteq \B$ be a unital inclusion of abelian $C^*$-algebras. By the Gelfand-Naimark Theorem, $\A \cong C(X)$, the continuous complex-valued functions on a compact Hausdorff space $X$ (cf. \cite[Theorem 1.7.3]{Arveson1976}). Likewise $\B \cong C(Y)$ for some compact Hausdorff space $Y$. Under these identifications, the inclusion map $\A \to \B$ corresponds to the map $C(X) \to C(Y): f \mapsto f \circ j$ for some continuous surjection $j:Y \to X$. It is easy to see that $\A \subseteq \B$ has the almost extension property if and only if
\[
    \overline{\{x \in X: \card(j^{-1}(x)) = 1\}} = X.
\]
By \cite[Corollary 3.21]{PittsZarikian2015}, $\A \subseteq \B$ has a unique pseudo-expectation if and only if there exists a unique minimal closed set $K \subseteq Y$ such that $j(K) = X$. Thus the abelian $C^*$-inclusion $C(\bbI) \subseteq C(\ddot{\bbI})$ corresponding to the continuous surjection $j:\ddot{\bbI} \to \bbI$ described in Example \ref{split interval} above has a unique pseudo-expectation, but not the almost extension property. On the other hand, by Theorem \ref{generalized Whyburn}, for an abelian $C^*$-inclusion $C(X) \subseteq C(Y)$ with $Y$ compact metric, the almost extension property and having a unique pseudo-expectation are equivalent. The metrizability of $Y$ corresponds precisely to the separability of $C(Y)$, and so Theorem \ref{generalized Whyburn} says that for separable abelian $C^*$-inclusions $\A \subseteq \B$, the almost extension property is equivalent to having a unique pseudo-expectation.\\

Recently, Kw\'{a}sniewski and Meyer substantially advanced our understanding of the relationship between the almost extension property and having a unique pseudo-expectation. Namely they introduced the notion of \emph{aperiodicity} for (arbitrary) $C^*$-inclusions \cite[Definition 5.14]{KM2019}, proved in \cite[Theorems 5.5 and 3.6]{KM2020} that
\begin{equation}
    \text{AEP} \implies \text{aperiodicity} \implies \text{$\exists!$ PsExp},
\end{equation}
and showed that for separable $C^*$-inclusions, the almost extension property is equivalent to aperiodicity \cite[Theorem 5.5]{KM2020}. In Theorem \ref{aperiodic} below, we prove that for abelian $C^*$-inclusions, aperiodicity is equivalent to having a unique pseudo-expectation. Combined with \cite[Theorem 5.5]{KM2020}, this gives another proof of Theorem \ref{generalized Whyburn} above, our generalization of Whyburn's Theorem.\\

For a $C^*$-algebra $\A$ we denote by $\A_1^+$ the positive norm-one elements of $\A$, and by $\bbH(\A)$ the non-zero hereditary subalgebras of $\A$.

\begin{definition}[aperiodic bimodules and $C^*$-inclusions]
Let $\A$ be a unital $C^*$-algebra. Following \cite{KM2019}:
\begin{itemize}
\item We say that a normed $\A$-bimodule $\X$ is \textbf{aperiodic} if for all $x \in \X$, $\varepsilon > 0$, and $\D \in \bbH(\A)$, there exists $d \in \D_1^+$ such that
\[
    \|dxd\| < \varepsilon.
\]
\item We say that an inclusion $\A \subseteq \B$ of unital $C^*$-algebras is \textbf{aperiodic} if $\B/\A$ is aperiodic as a normed $\A$-module. Equivalently, if for all $b \in \B$, $\varepsilon > 0$, and $\D \in \bbH(\A)$, there exists $d \in \D_1^+$ and $a \in \A$ such that
\[
    \|dbd - a\| < \varepsilon.
\]
\end{itemize}
\end{definition}

\begin{theorem} \label{aperiodic}
For an abelian $C^*$-inclusion $\A \subseteq \B$, aperiodicity is equivalent to having a unique pseudo-expectation.
\end{theorem}

\begin{proof}
Let $\A \subseteq \B$ be an abelian $C^*$-inclusion with unique pseudo-expectation $E:\B \to I(\A)$. Then $E$ is a $*$-homomorphism and $\J = \ker(E)$ is the unique maximal $\A$-disjoint ideal of $\B$ \cite[Corollary 3.21]{PittsZarikian2015}. To show that $\A \subseteq \B$ is aperiodic, we must show that $\B/\A$ is aperiodic as a normed $\A$-module. Since the canonical maps
\[
    (\J+\A)/\A \to \B/\A \to \B/(\J+\A)
\]
form an exact sequence of normed $\A$-modules, it suffices to show that $(\J+\A)/\A \cong \J$ and $\B/(\J+\A) \cong (\B/\J)/\A$ are aperiodic as normed $\A$-modules \cite[Lemma 5.12]{KM2019}.\\

The aperiodicity of the normed $\A$-module $(\B/\J)/\A$ is equivalent to the aperiodicity of the $C^*$-inclusion $\A \subseteq \B/\J$. That, in turn, follows from the aperiodicity of the $C^*$-inclusion $\A \subseteq I(\A)$ \cite[Remark 3.17]{KM2020} and the fact that the induced map $\dot{E}:\B/\J \to I(\A)$ is a $*$-monomorphism.\\

To show that $\J$ is aperiodic as a normed $\A$-module, we identify the $C^*$-inclusion $\A \subseteq \B$ with the $C^*$-inclusion $C(X) \subseteq C(Y)$ arising from the continuous surjection $j:Y \to X$. Then
\[
    \J = \{g \in C(Y): g|_K = 0\},
\]
where $K \subseteq Y$ is the unique minimal closed set such that $j(K) = X$. Let $g \in \J$, $\varepsilon > 0$, and $\D \in \bbH(\A)$. Then
\[
    \D = \{f \in C(X): f|_F = 0\},
\]
where $F \subsetneq X$ is a closed set. Define $U = \{y \in Y: |g(y)| < \varepsilon\}$, an open subset of $Y$ containing $K$. Assume that $j(U^c) \cup F = X$. Then $K' = U^c \cup j^{-1}(F)$ is a closed subset of $Y$ such that $j(K') = X$. It follows that $K \subseteq K'$, which implies $K \subseteq j^{-1}(F)$, which in turn implies $j(K) \subseteq F$, a contradiction. So $j(U^c) \cup F \subsetneq X$. By Urysohn's Lemma, there exists $f \in C(X)_1^+$ such that $f|_{j(U^c) \cup F} = 0$. Then $f \in \D_1^+$ and $\|(f \circ j)g(f \circ j)\| < \varepsilon$, which completes the proof.
\end{proof}

\bibliographystyle{amsplain}

\end{document}